\documentclass{article}
\usepackage{anyfontsize}
\usepackage[margin=1.25in]{geometry}
\usepackage{hyperref}
\usepackage{times,mathrsfs,mathtools,amsmath,amssymb}

\usepackage{amsthm}
\usepackage{stix}
\usepackage{titlesec}

\usepackage{parskip}

\begin{document}

\newtheorem{theorem}{\textsc{Theorem}}[section]
\newtheorem{problem}[theorem]{\textsc{Problem}}
\newtheorem{exercise}{\textsc{Exercise}}[section]
\newtheorem{proposition}[theorem]{\textsc{Proposition}}
\newtheorem{lemma}[theorem]{\textsc{Lemma}}
\newtheorem{corollary}[theorem]{\textsc{Corollary}}
\newtheorem{definition}[theorem]{\textsc{Definition}}
\newtheorem{remark}[theorem]{\rm \textsc{Remark}}
\newtheorem{example}[theorem]{\rm \textsc{Example}}

\renewcommand{\qedsymbol}{$\blacksquare$}

\newcommand{\bb}{\mathbb}
\newcommand{\mc}{\mathcal}
\renewcommand{\bf}{\mathbf}
\renewcommand{\bar}{\overline}
\renewcommand{\Re}{\text{Re}\,}
\renewcommand{\Im}{\text{Im}\,}
\newcommand{\im}{\text{im}\,}
\newcommand{\wtilde}{\widetilde}
\newcommand{\what}{\widehat}
\newcommand{\rhu}{\rightharpoonup}
\newcommand{\la}{\langle}
\newcommand{\ra}{\rangle}
\renewcommand{\r}{\right}
\renewcommand{\l}{\left}
\newcommand{\ind}{\text{ind}}
\newcommand{\res}{\text{Res}}
\newcommand{\bs}{\boldsymbol}
\newcommand{\tx}{\text}
\renewcommand{\v}{\tx{\bf v}}
\renewcommand{\u}{\tx{\bf u}}
\newcommand{\n}{\tx{\bf n}}
\newcommand{\w}{\tx{\bf w}}
\renewcommand{\div}{\text{div}\,}
\newcommand{\bdot}{\bs{\cdot}}
\newcommand{\on}{\operatorname}
\newcommand{\hooklongrightarrow}{\lhook\joinrel\longrightarrow}

\setcounter{section}{0}

\title{Asymptotics for the level set equation near a maximum}
\author{Nick Strehlke}
\date{}

\maketitle

\begin{abstract}
We give asymptotics for the level set equation for mean curvature flow on a convex domain near the point where it attains a maximum. It is known that solutions are not necessarily $C^3,$ and we recover this result and construct non-smooth solutions which are $C^3.$ We also construct solutions having prescribed behavior near the maximum. We do this by analyzing the asymptotics for rescaled mean curvature flow converging to a stationary sphere.
\end{abstract}

\section{Introduction}

Let $\Omega$ be a smooth bounded mean-convex domain in $\mathbb R^{n+1}.$ The level set equation on $\Omega$ is a degenerate elliptic boundary value problem asking for a function $t\colon \Omega \to\mathbb R$ with $t = 0$ on the boundary $\partial\Omega$ and 
\begin{align}
	|\nabla t| \div{\left(\frac{\nabla t}{|\nabla t|}\right)} = -1. \label{lse}
\end{align}
This problem is known to admit a unique, twice-differentiable solution that satisfies (\ref{lse}) in the classical sense away from critical points.  Away from critical points, the equation is non-degenerate elliptic and the solution is smooth. The second derivative, however, may in general be discontinuous at critical points.

If $t$ solves (\ref{lse}) for a mean convex domain, then the level sets $M_\tau = \{x\in \Omega\colon t(x) = \tau\}$ form a mean curvature flow starting from $M_0 = \partial\Omega,$ that is, the position vector $X(\tau)$ of $M_\tau$ satisfies
\begin{align*}
	N\bdot \frac{\mathrm d X}{\mathrm d\tau} &= - H,
\end{align*}
where $N$ is the outer unit normal for $M_\tau$ at the point $X$ and $H = \div_{M_\tau} N$ is the scalar mean curvature. Mean-convexity (meaning that the mean curvature $H$ of the boundary $\partial\Omega$ is nonnegative) is the condition required to ensure that the surfaces making up the mean curvature flow are disjoint. If $x\in \Omega,$ the value $t(x)$ is therefore the time at which the mean curvature flow starting from $M_0 = \partial \Omega$ arrives at the point $x.$ For this reason, the function $t$ is sometimes called the \emph{arrival time} for mean curvature flow.

If $\Omega$ is a bounded convex domain, it was proved by Huisken in \cite{Hu84} that the mean curvature flow $\{M_\tau\}$ starting from $\partial\Omega$ contracts smoothly to a single point $x_0\in \Omega$ at some finite time $T.$ Moreover, the translated and rescaled flow $(T-\tau)^{-1/2}(M_\tau-x_0)$ converges at time $T$ to the round sphere $\bf{S}^n$ of radius $(2n)^{1/2}$ centered at the origin. The function $t$ solving (\ref{lse}) for $\Omega$ therefore has a single critical point $x_0$ inside $\Omega,$ where $t(x_0) = T$ is the maximum for $t.$ In this case, $t$ is actually $C^2$ on $\Omega$ and the second derivative $\nabla^2 t(x_0)$ of $t$ at this critical point is the identity: $\partial_i\partial_j t = \delta_{ij}.$\footnote{See Theorem 6.1 of \cite{Hu93}.}

In the case of a general mean-convex domain, the arrival time $t$ is known to be twice differentiable but not necessarily $C^2,$ see \cite{CoMi16}, \cite{CoMi17}, and \cite{CoMi18}. In fact, it was shown in \cite{Wh00} (Theorem 1.2) and \cite{Wh11} that any tangent flow of a smooth mean convex mean curvature flow is a generalized cylinder. From this one can figure out what the Hessian of the arrival time function must be if it exists. The remaining issue was to show that the Hessian exists, which is equivalent to the problem of uniqueness of tangent flows. This was solved in \cite{CoMi15}. The study of the arrival time is referred to as the \emph{level set method} in the mean curvature flow literature, because it gives a means of rigorously extending mean curvature flow beyond singularities. This point of view was first taken in a computational context by Osher and Sethian, \cite{OsSe88}, and the theory was then developed in \cite{ChGi91}, \cite{EvSp91}, \cite{EvSp92}, \cite{EvSp92a}, and \cite{EvSp95}. We will restrict attention to the case in which the domain of the arrival time function is convex.

In \cite{KoSe06}, Robert Kohn and Sylvia Serfaty proved that the solution to equation (\ref{lse}) on a convex planar domain $\Omega$ is always $C^3,$ and they asked whether this is true in higher dimensions. Natasa Sesum demonstrated in \cite{Se08} that the answer is negative: if $n\geq 2,$ there exists a convex domain $\Omega\subset \mathbb R^{n+1}$ for which the solution $t$ to (\ref{lse}) is not three times differentiable. To prove this, she analyzed the rate of convergence of a rescaled MCF $(T-\tau)^{-1/2}M_\tau,$ proving the existence of solutions for which this rescaled flow converges to the sphere like $(T-\tau)^{1/n}.$ 

We recover this result and extend it by describing all possible rates of convergence for rescaled MCF over the sphere. As a result, we are able to describe the first terms of all possible Taylor expansions of a solution $t$ to equation (\ref{lse}) on a convex domain $\Omega$ at the point where $t$ attains its maximum. We also construct solutions which have the prescribed asymptotics, but we do not prove here that they are actually Taylor expansions (we do not show that the solutions are better than $C^2$ on the domain $\Omega$). The main result is the following theorem.

\bigskip
\begin{theorem}
\label{arrival-time}
Let $t$ be a solution to the level set equation (\ref{lse}) on a smooth bounded convex domain $\Omega\subset\mathbb R^{n+1}$ which attains its maximum $T$ at the the origin. Then either $\Omega$ is a round ball and $t = T-|x|^2/(2n)$ for $x\in \Omega,$ or there exists an integer $k\geq 2$ and a nonzero homogeneous harmonic polynomial $P$ of degree $k$ for which $t$ has, at the origin, the asymptotic expansion
	\begin{align}
		t(x) &= T- \frac{|x|^2}{2n} + |x|^{k(k-1)/n}P(x) + O\left(|x|^{\sigma + k + k(k-1)/n}\right) \label{taylor}
	\end{align}
	for some $\sigma>0.$ Moreover, if $P$ is a homogeneous harmonic polynomial of degree $k$ there exists a smooth bounded convex domain $\Omega \subset\mathbb R^{n+1}$ for which the corresponding arrival time $t$ satisfies (\ref{taylor}) near the origin where it attains its maximum $T.$
\end{theorem}

\emph{Remark.} Part of the statement of the theorem is a unique continuation result for the arrival time on a convex domain: if the arrival time attains its maximum at the origin and coincides to infinite order there with the arrival time $T-|x|^2/(2n)$ for a ball, then in fact the domain is a ball and the arrival time is identically equal to $T-|x|^2/(2n).$ This is proved in a companion paper, \cite{Str18}, as a consequence of the fact that a rescaled mean curvature flow cannot converge to a sphere faster than any exponential unless it is identically equal to the sphere.\footnote{A different and more complicated parabolic unique continuation property for self-shrinkers was recently proved by Jacob Bernstein in \cite{Be17}.}

As will be seen in Section \ref{as-arr}, Theorem \ref{arrival-time} follows straightforwardly from Theorems \ref{stable-mfld} and \ref{prescribed-lin} on the possible rates of convergence for rescaled mean curvature flow. As a consequence, the statement of the asymptotic expansion (\ref{taylor}) in Theorem \ref{arrival-time} may be sharpened in keeping with the slightly more complicated statement of Theorem \ref{stable-mfld}. The most precise statement is: Let $\lambda_j = j(j+n-1)/(2n)-1$ be the $j$th eigenvalue for the operator $\Delta+1$ on the sphere $\bf{S}^n$ of radius $(2n)^{1/2}$ centered at the origin. For $j\geq k$ such that $\lambda_j<2\lambda_k,$ there exists a homogeneous harmonic polynomial $P_j$ of degree $j$ such that
\begin{align*}
	t(x) &= T- \frac{|x|^2}{2n} + \sum_{\genfrac{}{}{0pt}{2}{j\geq k}{j + j(j-1)/n<2k + 2k(k-1)/n}}|x|^{j(j-1)/n}P_j(x) + O\left(|x|^{2\sigma}\right)
\end{align*}
for all $\sigma < k + k(k-1)/n.$ Notice that $j+j(j-1)/n = 2+2\lambda_j.$

In particular, when $k\geq 3$ or $n=1$ or $2,$ the exponent $\sigma$ appearing in (\ref{taylor}) can be taken equal to $1.$ If $k=2$ and $n\geq 3,$ then we can choose any $\sigma <2/n.$

\medskip
We do not prove in this paper that the asymptotic expansion (\ref{taylor}) of Theorem \ref{arrival-time} is actually a Taylor expansion, though of course it is true that the Taylor expansion at the origin must coincide with (\ref{taylor}) \emph{if it exists}. Proving existence requires bounding the derivative of the arrival time in a neighborhood of the origin, an analysis we do not carry out here. It follows, however, from results of Huisken and Sesum,\footnote{See Theorem 6.1 of \cite{Hu93} for Huisken's proof that the arrival time is $C^2$ and Corollary 5.1 of \cite{Se08} for Sesum's proof that the arrival time is $C^3$ in case $k\geq 3$ in our Theorem \ref{arrival-time}.} that the arrival time for a convex domain is $C^2$ in all cases and that, in case $k\geq 3$ in our Theorem \ref{arrival-time}, the arrival time is $C^3.$

\medskip
In the following section, we introduce the rescaled mean curvature flow and describe the relationship between rates of convergence for rescaled MCF and the Taylor expansion of the arrival time near its maximum.

\section{Rate of convergence of MCF and relation to level set equation}
\label{as-arr}

We begin by recalling the rescaled mean curvature flow. Let $\Omega$ be a convex domain and let $\{M_\tau\}_{\tau\in [0,T)}$ be the mean curvature flow starting from $M_0 = \partial\Omega.$ As mentioned in the introduction, $M_\tau$ shrinks smoothly down to a point $x_0\in \Omega$ as $\tau\to T$ in such a way that the rescaled surfaces $(T-\tau)^{-1/2}(M_\tau - x_0)$ converge in $C^k,$ for any $k,$ to the sphere $\bf{S}^n$ of radius $(2n)^{1/2}$ centered at the origin in $\mathbb R^{n+1}.$ 

It is natural therefore to study the surfaces $(T-\tau)^{-1/2} (M_\tau - x_0),$ and the analysis is simplified by a change of variable: we put $s = -\log{(T-\tau)}$ and for $s\geq -\log{T}$ define $\Sigma_s =  (T-\tau)^{-1/2}(M_\tau - x_0) = e^{s/2} (M_{T-e^{-s}} - x_0).$ The $1$-parameter family $\{\Sigma_s\}$ is called a \emph{rescaled mean curvature flow}. Its position vector $X(s)$ satisfies the equation
\begin{align*}
	\frac{\mathrm d X}{\mathrm ds} \bdot N = -H + \frac{1}{2}X\bdot N,
\end{align*}
with $N$ and $H$ now the outer unit normal and scalar mean curvature of $\Sigma_s.$ The sphere $\bf{S}^n$ of radius $(2n)^{1/2}$ centered at the origin is stationary under the rescaled mean curvature flow.\footnote{Surfaces that are stationary for this flow are in general called \emph{self-shrinkers}, because they shrink homothetically under mean curvature flow. It was shown in \cite{Br16} that the sphere $\bf{S}^n$ is the only compact embedded self-shrinker with genus zero.}

Let $\bf{n}$ be the outer unit normal to the sphere $\bf{S}^n.$  If $\{\Sigma_s\}$ is a convex rescaled mean curvature flow, then it converges as $s\to\infty$ to $\bf{S}^n$ in $C^2.$ This means that there exists $s_0\in \mathbb R$ and a scalar function $u\colon \bf{S}^n\times[s_0,\infty)\to \mathbb R$ with the property that $\Sigma_s$ is the normal graph of $u(\cdot,s)$ over the sphere $\bf{S}^n$ for $s\geq s_0$:
\begin{align*}
	\Sigma_s = \left\{y + u(y,s)\bf{n}(y)\colon y\in \mathbf{S}^n\right\}.
\end{align*}
The function $u$ is uniquely determined and solves a quasilinear parabolic equation
\begin{align}
	\partial_s u &= \Delta u + u + N(u,\nabla u,\nabla^2u) \label{rMCF}
\end{align}
where $\Delta$ is the Laplacian on $\bf{S}^n$ and $N$ is a nonlinear term of the following form:
\begin{align*}
	N(u,\nabla u,\nabla^2 u) = f(u,\nabla u) + \on{trace}(B(u,\nabla u)\nabla^2u),
\end{align*}
where $f$ and $B$ are smooth and $f(0,0),$ $df(0,0),$ and $B(0,0)$ are zero.

We now state our results on the rate of convergence for rescaled mean curvature for a sphere. Our first main result is that a solution to the equation (\ref{rMCF}) that converges to zero as $s\to\infty$ approaches a solution to the linear equation $\partial_s u = \Delta u + u.$ 

The linear operator $\Delta + 1$ has discrete spectrum with eigenvalues $\lambda_k = k(k+n-1)/(2n)-1,$ for $k=0,1,2,\dots,$ each corresponding to an eigenspace $E_k$ of finite dimension $\binom{n+k}{n}-\binom{n+k-2}{n}.$ Notice that zero is not an eigenvalue. Let $d_k$ be the dimension of the space of eigenfunctions corresponding to eigenvalues $\lambda_0,\dots,\lambda_{k-1}$ which are strictly smaller than $\lambda_k.$

With this notation, the precise result is the following.

\medskip
\begin{theorem}
\label{stable-mfld}
	For any integer $r>n/2+1$ and any integer $k\geq 2,$ there exists an open neighborhood $B = B(k,r)$ of the origin in $H^r(\bf{S}^n)$ with the property that the set of initial data $u_0\in B$ for which the solution $u$ to the rescaled MCF equation (\ref{rMCF}) exists for all time $s\geq 0$ and converges to zero with exponential rate $\lambda_k$ is a codimension $d_k$ submanifold of $B$ which is invariant for equation (\ref{rMCF}). For such initial data, there exist, for $j\geq k$ with $\lambda_j<2\lambda_k,$ eigenfunctions $P_j \in E_j$ for which the corresponding solution $u$ satisfies
	\begin{align*}
		\big\| u(y,s) - \sum_{\genfrac{}{}{0pt}{2}{j\geq k}{\lambda_j<2\lambda_k}} e^{-\lambda_j s} P_j(y)\big\|_{H^r(\mathbf{S}^n)} \leq C e^{-2\sigma s}
	\end{align*}
	for some constant $C>0$ and all $\sigma<\lambda_k.$
\end{theorem}

\emph{Remark.} The proof closely follows the proof of the analogous theorem for ODEs. Moreover, the proof of the existence of an invariant manifold is modeled on the argument of \cite{Na88} (which generalizes \cite{EpWe87}).

\medskip
We also prove that the leading eigenfunction $P_k$ to which $e^{\lambda_k s}u(x,s)$ converges in $H^r(\bf{S}^n)$ may be prescribed.

\medskip
\begin{theorem}
\label{prescribed-lin}
Suppose $k\geq 2$ and let $P \in E_k$ be an eigenfunction for the operator $\Delta+1$ on the sphere $\bf{S}^n$ corresponding to the eigenvalue $\lambda_k.$ There exists $s_0\geq0$ and $u\colon \mathbf{S}^n\times[s_0,\infty)\to\mathbb R$ which solves the rescaled MCF equation (\ref{rMCF}) and satisfies
\begin{align*}
	\|e^{\lambda_k s} u(y,s) - P(y)\|_{H^r(\bf{S}^n)} \leq Ce^{-\sigma s}
\end{align*}
for some constants $C>0$ and $\sigma>0$ and for all $s\geq s_0.$  
\end{theorem}

\emph{Remarks.} If $k\geq 3$ or $n=1$ or $2,$ then we may take $\sigma = \lambda_{k+1}$ in the statement of the theorem, and if $n\geq 3$ and $k=2$ we may take any $\sigma <2\lambda_2 = 2/n.$

The precise asymptotics of the limit, and the prescription of them, are inspired by \cite{AnVe97}. In fact, the present investigation came from the author's wish to determine similar asymptotics in the simpler compact setting.

\medskip
We now show the relationship between these results and the level set equation. We will derive Theorem \ref{arrival-time} from Theorems \ref{stable-mfld}(a) and \ref{prescribed-lin}.

Let $\Omega\subset\mathbb R^{n+1}$ be a bounded convex region and suppose $t\colon \Omega \to \mathbb R$ with $t(x) = 0$ on $\partial\Omega$ solves the level set equation (\ref{lse}) on $\Omega.$ Let $M_\tau = \{x\in \Omega\colon t(x) = \tau\}$ be the corresponding mean curvature flow and $\Sigma_s$ the corresponding rescaled MCF. Then $\Sigma_s$ converges to the sphere $\bf{S}^n$ as $s\to\infty,$ and, as remarked previously, it follows that for sufficiently large $s$ the surface $\Sigma_s$ is a normal graph over $\bf{S}^n$: there exists $s_0\geq0$ and a function $u\colon \bf{S}^n\times [s_0,\infty)\to\mathbb R$ which solves (\ref{rMCF}) and for which
\begin{align*}
	\Sigma_s &= \{y+ u(y,s)\bf{n}(y)\colon y\in\bf{S}^n\}.
\end{align*}
By rescaling the initial mean curvature flow if necessary, we may take $s_0=0$ without loss of generality.

By Theorem \ref{stable-mfld}, either $u$ is identically zero or there exists $k\geq 2$ and a nonzero homogeneous harmonic polynomial $P$ of degree $k,$ the restriction of which to $\bf{S}^n$ is an eigenfunction in $E_k$ corresponding to the eigenvalue $\lambda_k$ of $\Delta + 1,$ for which
\begin{align}
	u(y,s) &= e^{-\lambda_k s} P(y) + O\left(e^{-(\lambda_k + \sigma) s}\right) \label{asymptotic1}
\end{align}
in $H^{r+1}(\mathbf{S}^n)$ as $s\to\infty,$ where $\sigma>0.$ Since $r>n/2+1$ this bound actually holds in $L^\infty(\mathbf{S}^n)$ by the Sobolev imbedding theorem.

The position vector of a point $x$ of $M_t = (T-t)^{1/2}\Sigma_s$ must satisfy the equation
\begin{align*}
	x &= (T-t)^{1/2}\frac{x}{|x|}(2n)^{1/2} + (T-t)^{1/2}u\left(\frac{x}{|x|}, s\right) \frac{x}{|x|},
\end{align*}
remembering as always that $s = -\log{(T-t)}.$ In other words,
\begin{align*}
	\frac{|x|}{(2n)^{1/2}} &= (T-t)^{1/2}\left(1 + u\left(\frac{x}{|x|}, s\right)\right). 
\end{align*}

 Substituting the asymptotic (\ref{asymptotic1}) for $u$ and replacing $s$ with $-\log{(T-t)}$ gives
\begin{align*}
	\frac{|x|}{(2n)^{1/2}} &= (T-t)^{1/2}\left(1 + e^{-\lambda_k s} P\left(\frac{x}{|x|}\right) + O\left(e^{-(\lambda_k + \sigma) s}\right)\right)\nonumber\\
	&= (T-t)^{1/2}\left(1 + (T-t)^{\lambda_k }|x|^{-k}P(x) + O\left( (T-t)^{\lambda_k + \sigma}\right)\right). 
\end{align*}
Since it is known that $T-t\to 0$ as $x\to0,$ this equation implies that $T-t = |x|^2/(2n) + o(|x|^2)$ as $x\to0.$ But then squaring and rearranging and substituting this for $T-t$ we obtain
\begin{align}
	T-t &= \frac{|x|^2}{2n} - (T-t)^{1+\lambda_k} |x|^{-k}P(x) + O\left((T-t)^{1+\lambda_k + \sigma}\right) \label{asymptotic2}\\
	&= \frac{|x|^2}{2n} - \left(\frac{|x|^2}{2n} + o\left(|x|^2\right)\right)^{1+\lambda_k} |x|^{-k}P(x) + O\left( |x|^{2 + 2\lambda_k + 2\sigma}\right)\nonumber\\
	&= \frac{|x|^2}{2n} - \frac{|x|^{2+2\lambda_k - k}}{(2n)^{1+\lambda_k}} P(x) + o\left(|x|^{2+2\lambda_k}\right). \label{asymptotic3}
\end{align}
Finally, substituting this improved asymptotic (\ref{asymptotic3}) for each occurrence of $T-t$ in the first line (\ref{asymptotic2}) and carrying out the same computation gives the improvement
\begin{align*}
	T-t &= \frac{|x|^2}{2n} - \frac{|x|^{2+2\lambda_k - k}}{(2n)^{1+\lambda_k}} P(x) + O\left(|x|^{2+4\lambda_k}+|x|^{2 + 2\lambda_k + 2\sigma}\right),
\end{align*}
which is equivalent to the conclusion of Theorem \ref{arrival-time}.

\medskip

In the remainder of the paper, we prove the results of Theorems \ref{stable-mfld} and \ref{prescribed-lin}. We prove Theorem \ref{stable-mfld} in the next section and afterward prove \ref{prescribed-lin}.

\section{Construction of the invariant manifolds}
\label{inv-man}

In this section, we adapt the argument of \cite{Na88}, which is a general stable manifold theorem for geometric evolution equations, to our situation in order to construct invariant manifolds of solutions which converge with prescribed exponential rate. We now briefly summarize the main result of \cite{Na88} and explain how our results differ: Let $M$ be a closed Riemannian manifold of dimension $n$ and let $L$ be an elliptic differential operator on $M$ which is symmetric in the $L^2(M)$ inner product and which has discrete spectrum accumulating only at $+\infty$ (in particular the operator is assumed to be bounded below). Suppose $N = N(u)$ is a nonlinear function defined on $H^{r-1}(M)$ for an integer $r>n/2+1$ which satisfies $N(0) = 0$ and a bound of the form we prove in Lemma \ref{bilin-bd}. In this situation, Naito proves the following:

\medskip
\begin{theorem}[(Naito, \cite{Na88})]
	\label{na-thm} There exists a ball $B$ centered at the origin in $H^{r+1}(M)$ in which the nonlinear evolution equation
	\begin{align*}
		\partial_su &= Lu + N(u)
	\end{align*}
	has an invariant stable manifold of finite codimension.
\end{theorem}

The codimension is equal to the codimension of the space on which $L$ is negative definite (the index of $L$ plus the dimension of the kernel). Naito's argument is modeled on Epstein \& Weinstein's earlier proof of a stable manifold theorem for mean curvature flow in the plane, \cite{EpWe87}, and both of these arguments follow closely the proof of the stable manifold theorem for ODE.\footnote{For a treatment of the stable manifold theorem in the finite-dimensional ODE context, see, e.g., \cite{Ha09}, \S III.6.}

Theorem \ref{na-thm} already almost implies part of the conclusion of Theorem \ref{stable-mfld}, though it does not include the precise rate of convergence and does not describe the asymptotics of the limit. Using the notation of Theorem \ref{stable-mfld} from the preceding subsection and assuming $k\geq 2,$ one would like, in our situation, to replace a solution $u(x,s)$ of (\ref{rMCF}) with $e^{\lambda_{k} s}u(x,s)$ and to replace the linear term $\Delta + 1$ on the right side of (\ref{rMCF}) with $L = \Delta + 1 + \lambda_k$ and then to apply Naito's theorem. The main issue then is that the nonlinear term will depend on the time parameter $s,$ but this is easy to overcome in this context because the time-dependent nonlinear term satisfies a bound that is uniform in $s.$ 

Notice that, assuming this argument is carried out successfully, the stable manifold one obtains in this case from Theorem \ref{na-thm} is the set of solutions for which $e^{\lambda_k s}u(s)\to 0,$ and it will have the codimension of all eigenspaces corresponding to eigenvalues $\lambda_j$ with $j\leq k$ (the index plus nullity of $\Delta+1+\lambda_k$). If we want precisely the solutions for which $s\mapsto e^{\lambda_k s}u(x,s)$ is bounded, that is, precisely the solutions for which $u$ converges to $0$ exponentially at rate $\lambda_k$ as $s\to\infty,$ we must instead apply Theorem \ref{na-thm} to $e^{(\lambda_k -\varepsilon)s} u(x,s)$ and $L = \Delta + 1 +\lambda_k - \varepsilon$ for sufficiently small $\varepsilon.$ The ultimate conclusion of this analysis is that there exists a codimension $d_k$ invariant submanifold for the equation (\ref{rMCF}) with the property that any solution in this invariant submanifold converges to zero at exponential rate $\lambda_k - \varepsilon$ for all $\varepsilon>0.$ In particular, this argument does not prove that $e^{\lambda_ks}u(s)$ is bounded in $H^{r+1}(M),$ though this can be proved (and we prove it below) using the bound on the nonlinear term.\footnote{The assertion is not true for a general nonlinear term, as is already apparent in the finite-dimensional ODE case, for essentially the same reason that a center manifold need not be stable. Consider for example the ODE
\begin{align*}
	\dot{x} = -\varepsilon x - \frac{x}{\log{|x|}}
\end{align*}
on $\mathbb R.$ For small initial data, the solution converges to zero like $te^{-\varepsilon t}$ as $t\to\infty.$ If the nonlinear term is $O(|x|^{1+\alpha})$ for some $\alpha>0$ as $x\to0$ this cannot happen.} Thus the bound on the nonlinear term does imply that the rate of convergence is better than shown in \cite{Na88} or \cite{EpWe87}.\footnote{Cf. Proposition 5.2 of \cite{Na88}, where the author establishes convergence to zero with exponential rate $\sigma$ for any $\sigma$ smaller than the first positive eigenvalue of the linear operator, and Remark 3, page 136 of \cite{EpWe87}, where the same claim is made.} The same argument improves the rate of convergence in Naito's general theorem, because we only use his bound on the nonlinear term.

\medskip
Rather than apply the conclusion of Theorem \ref{na-thm} in this way, we prefer to adapt the argument to our situation. This is done in this section (Section \ref{inv-man}). Section \ref{lin-an} collects some bounds required for the construction in Section \ref{contraction}, and both sections follow closely arguments of \cite{Na88} and \cite{EpWe87}. We also include, for the convenience of the reader, a proof that a quasilinear nonlinear term $N$ of second order does satisfy the bound required by Naito's hypotheses in \cite{Na88} and Theorem \ref{na-thm}. This occupies Section \ref{nonlin-est}

In Section \ref{asymptotics}, we establish the rest of Theorem \ref{stable-mfld}, namely, the precise rate of convergence and the asymptotics. This part does not overlap with \cite{Na88} or \cite{EpWe87}, and in fact the same arguments extend the results of \cite{Na88} in the more general setting of that paper. We also show that the asymptotics can be prescribed as in Theorem \ref{prescribed-lin}. Analysis of the asymptotics requires a closer look at the construction of the stable invariant manifold in the first place, and this is part of the reason we prefer to argue directly in the proof of Theorem \ref{stable-mfld} rather than attempt to apply the conclusion of Theorem \ref{na-thm} to our situation.

\subsection{Linear estimates}
\label{lin-an}

Throughout, we write $\langle v,w\rangle$ for the $L^2(\mathbf{S}^n)$ inner product:
\begin{align*}
	\langle v,w\rangle = \int_{\mathbf{S}^n} vw.
\end{align*}

 Let $L$ be the linear operator $\Delta + 1$ on the sphere $\mathbf{S}^n,$ and let $F_k$ be the subspace of $H^r(\mathbf{S}^n)$ defined by
\begin{align*}
	F_k &= \bigoplus_{j=k}^\infty E_j
\end{align*}
with $E_j$ as before the eigenspace for $L$ corresponding to the $j$th eigenvalue $\lambda_j = j(j+n-1)/(2n) - 1.$ From now on, we fix an integer $k\geq 2$ so that $L$ is negative definite and bounded above on $F_k,$ satisfying $\langle Lv,v\rangle \leq -\lambda_k\|v\|_{L^2(\mathbf{S}^n)}^2$ for $v\in F_k.$

For $v\in F_k,$ we may \emph{define} the $H^\ell(\mathbf{S}^n)$ norm for integer $\ell\geq 0$ by
\begin{align*}
	\|v\|_{H^\ell} : = \langle (-L)^\ell v,v\rangle.
\end{align*}
This norm is equivalent to the usual $H^\ell$ norm.
\medskip 

\begin{lemma}
	If $s\mapsto v(s)$ is a continuously differentiable path in $F_k\cap H^{r+1}(\mathbf{S}^n),$ then for any $\varepsilon>0$ and any integer $r\geq 1,$
	\begin{align}
		\frac{1}{2}\frac{\mathrm d}{\mathrm ds} \|v(s)\|_{H^r(\mathbf{S}^n)}^2 + (1-\varepsilon) \|v(s)\|_{H^{r+1}(\mathbf{S}^n)}^2 \leq \frac{1}{4\varepsilon} \|(\partial_s - L)v(s)\|_{H^{r-1}(\mathbf{S}^n)}^2. \label{lin-est1}
	\end{align}
\end{lemma}
\begin{proof}
	Write $f = (\partial_s - L)v$ for brevity. Use Cauchy-Schwarz to get, for any $\varepsilon>0,$
	\begin{align*}
		\langle (-L)^r v, f\rangle &\leq \varepsilon \langle (-L)^{r+1}v,v\rangle + \frac{1}{4\varepsilon} \langle (-L)^{r-1} f,f\rangle.
	\end{align*}
	Rearranging and substituting $f = (\partial_s - L)v$ on the left gives
	\begin{align*}
		\langle (-L)^r v, (\partial_s - (1-\varepsilon)L)v \rangle \leq \frac{1}{4\varepsilon}\langle (-L)^{r-1}f,f\rangle = \frac{1}{4\varepsilon} \|f\|_{H^{r-1}}^2,
	\end{align*}
	and because $\partial_s \|v\|_{H^r}^2/2 = \langle (-L)^r v, \partial_s v\rangle$ and $\langle (-L)^{r+1}v,v\rangle = \|v\|_{H^{r+1}}^2,$ this is equivalent to the conclusion of the lemma.
\end{proof}

\begin{corollary}
	\label{sup-bd}
	If $v(s)\in F_k$ for all $s\geq 0,$ then for any $\sigma$ with $0<\sigma<\lambda_k$ and any integer $r\geq 1,$
	\begin{align*}
		e^{2\sigma s}\|v(s)\|_{H^r(\mathbf{S}^n)}^2 \leq \|v(0)\|_{H^r(\mathbf{S}^n)}^2 + \frac{\lambda_k}{2(\lambda_k - \sigma)}\int_0^s e^{2\sigma\tau}\|(\partial_s - L)v(\tau)\|_{H^{r-1}(\mathbf{S}^n)}^2\,\mathrm d\tau.
	\end{align*}
\end{corollary}
\begin{proof}
	Notice that the left side of (\ref{lin-est1}) can be bounded below for $v\in F_k$ using $\|v(s)\|_{H^{r+1}}^2\geq \lambda_k\|v(s)\|_{H^r}^2.$ The result is
	\begin{align*}
		\frac{\mathrm d}{\mathrm ds} \|v(s)\|_{H^r(\mathbf{S}^n)}^2 + 2(1-\varepsilon)\lambda_k \|v(s)\|_{H^{r+1}(\mathbf{S}^n)}^2 &\leq \frac{1}{2\varepsilon} \|(\partial_s - L)v(s)\|_{H^{r-1}(\mathbf{S}^n)}^2.
	\end{align*}
	This is equivalent to the statement of the corollary with $\sigma = (1-\varepsilon)\lambda_k$ because the left side can be written 
	\begin{align*}
		\frac{\mathrm d}{\mathrm ds} \|v(s)\|_{H^r(\mathbf{S}^n)}^2 + 2(1-\varepsilon)\lambda_k \|v(s)\|_{H^{r+1}(\mathbf{S}^n)}^2&= e^{-2(1-\varepsilon)\lambda_ks} \frac{\mathrm d}{\mathrm ds} \left(e^{2(1-\varepsilon)\lambda_ks}\|v(s)\|_{H^r}^2\right)
	\end{align*}
	and we can multiply through by $e^{2(1-\varepsilon)\lambda_ks}$ and integrate.
\end{proof}

\begin{corollary}
	\label{L2-bd}
	In the situation of the lemma, if $r\geq 1$ is an integer and $\|v(s_j)\|_{H^r(\mathbf{S}^n)}\to 0$ for some sequence $s_j$ increasing to infinity, then
	\begin{align*}
		\int_0^\infty \|v(s)\|_{H^{r+1}(\mathbf{S}^n)}^2\,\mathrm ds &\leq \|v(0)\|_{H^r(\mathbf{S}^n)} + \int_0^\infty \|(\partial_s - L)v(s)\|_{H^{r-1}(\mathbf{S}^n)}^2\,\mathrm ds.
	\end{align*}
\end{corollary}

\subsection{Nonlinear estimate}
\label{nonlin-est}

The nonlinear term $N\colon \mathbb R\times \Gamma(T\mathbf{S}^n) \times \Gamma(T^* \mathbf{S}^n\otimes T\mathbf{S}^n)\to\mathbb R$ (here $\Gamma(T\mathbf{S}^n)$ is the space of sections of the tangent bundle, for instance) appearing in the rescaled mean curvature flow equation (\ref{rMCF}) over the sphere has the form
\begin{align}
	N(u,\nabla u,\nabla^2 u) &= f(u,\nabla u) + \on{trace}(B(u,\nabla u)\nabla^2u) \label{quasilin-assump}
\end{align}
where $f\colon \mathbb R\times \Gamma(T\mathbf{S}^n) \to\mathbb R$ is smooth with $f(0,0) = 0$ and $Df(0,0) = 0,$ and where $B\colon \mathbb R\times \Gamma(T\mathbf{S}^n) \to \Gamma(T^*\mathbf{S}^n\otimes T\mathbf{S}^n)$ is smooth and satisfies $B(0,0) = 0.$\footnote{See \cite{CoMi15}, Appendix A, for a proof of this fact.}

In this section, we prove the following Sobolev estimate for a nonlinear term $N$ of this form. We abbreviate $N(u,\nabla u,\nabla^2 u)$ by $N(u).$

\medskip
\begin{lemma}
	\label{bilin-bd}
	 Let $r$ be an integer with $r>n/2+1,$ let $N$ be smooth function of the form (\ref{quasilin-assump}), and let $R>0$ be fixed. There exists a constant $C$ depending on $N$ and $R$ and $r$ with the property that all $v,w\in C^\infty(\mathbf{S}^n)$ with $\|v\|_{H^r(\mathbf{S}^n)},\|w\|_{H^r(\mathbf{S}^n)}\leq R$ satisfy
	\begin{align*}
		\|N(v) - N(w)\|_{H^{r-1}(\mathbf{S}^n)} \leq C\left( \|v\|_{H^{r+1}(\mathbf{S}^n)}\|v-w\|_{H^r(\mathbf{S}^n)} + \|w\|_{H^r(\mathbf{S}^n)} \|v-w\|_{H^{r+1}(\mathbf{S}^n)}\right).
	\end{align*}
\end{lemma}

For the proof of Lemma \ref{bilin-bd}, we need a Sobolev product lemma which is standard. In this simple case ($s$ an integer) it can be proved using H\"older's inequality and the Sobolev imbedding theorems.
\medskip
\begin{lemma}
	\label{sobolev-prod}
	 Suppose $M = M^n$ is a closed Riemannian manifold of dimension $n,$ and $s_1,s_2,$ and $s$ satisfy $s_i\geq s$ and $s_1+s_2\geq s+d/2.$ Then there is a constant $C$ depending on $s$ and the Sobolev constant for $M$ such that
	\begin{align*}
		\|vw\|_{H^s(M)} \leq C\|v\|_{H^{s_1}(M)}\|w\|_{H^{s_2}(M)}
	\end{align*}
	for all $v,w\in C^\infty(M).$
\end{lemma}

We now indicate the proof of Lemma \ref{bilin-bd}, demonstrating the bound on the $f$ term of $N.$ The other term is similar so we omit the details. For clarity, let us now work in a coordinate chart (it makes no difference in the analysis). Thus let $u_j = \partial_ju$ be the components of the gradient $\nabla u.$ Under the preceding assumptions, we can express $f$ as
\begin{align*}
	f(u,\nabla u) &= g_0(u,\nabla u) u^2 + \sum_{j=1}^n g_j(u,\nabla u) u_j^2
\end{align*}
for some smooth functions $g_j.$ In particular,
\begin{align*}
	f(u,\nabla u) - f(v,\nabla v) &= g_0(u,\nabla u) (u-v)(u+v) + (g_0(u,\nabla u) - g_0(v,\nabla v))v^2 \\
	&\quad + \sum_{j=1}^n g_j(u,\nabla u) (u_j - v_j)(u_j+v_j) + (g_j(u,\nabla u) - g_j(v,\nabla v))v_j^2. 
\end{align*}
Now suppose that $u$ and $v$ are in $H^{r}(\mathbf{S}^n),$ where $r>n/2+1.$ There is a continuous imbedding $H^r(\mathbf{S}^n) \hooklongrightarrow C^1(\mathbf{S}^n),$ and so the $C^1$ norms of $u$ and $v$ are controlled by the $H^r$ norms. In this situation, if we assume that $\|u\|_{H^r},\|v\|_{H^r}\leq R,$ we can deduce that the functions $g_j(u,\nabla u)$ satisfy
\begin{align*}
	\|g_j(u,\nabla u)\|_{H^\ell}\leq C(1 + \|u\|_{H^{\ell+1}})
\end{align*}
for any integer $\ell\geq 0,$ where $C$ is a constant that depends on the function $g_j$ and on $R.$ (The proof is by induction, and we use the fact that the domain $\mathbf{S}^n$ has finite volume.) In particular, $g_j(u,\nabla u)$ and $g_j(v,\nabla v)$ are in $H^{r-1},$ and since $r-1>n/2$ we may apply the Sobolev product theorem (with $r-1=s=s_1=s_2$) to terms like $g_j(u,\nabla u) (u_j - v_j)(u_j+v_j).$  To deal with the terms $(g_j(u,\nabla u) - g_j(v,\nabla v))v_j^2,$ we write
\begin{align*}
	g_j(u,\nabla u) - g_j(v,\nabla v) &= \int_0^1 \partial_1 g(u + t(v-u),\nabla u + t\nabla(v-u))\,\mathrm dt\, (v-u) \\
	&\qquad + \sum_{i=2}^{n+1} \int_0^1 \partial_i g(u + t(v-u),\nabla u + t\nabla (v-u))\,\mathrm dt\, (v_i - u_i).
\end{align*}
The functions $\int_0^1 \partial_i g(u + t(v-u),\nabla u + t\nabla (v-u))\,\mathrm dt$ are in $H^{r-1}$ for the same reason that $g_j(u,\nabla u)$ is, and so we may apply the Sobolev product theorem to these terms as well.

Combining everything, we get a bound
\begin{align*}
	\|f(u,\nabla u) - f(v,\nabla v)\|_{H^{r-1}} &\leq \|(u-v)(u+v)\|_{H^{r-1}} + C\sum_{j=1}^n\|(u_j-v_j)(u_j+v_j)\|_{H^{r-1}} \\
	&\qquad + C\|(u-v)(v^2+  |\nabla v|^2)\|_{H^{r-1}} +  C\sum_{j=1}^n\|(u_j-v_j)(v^2 + |\nabla v|^2)\|_{H^{r-1}}
\end{align*}
where the constant $C$ depends on $f$ and $R$ and $r.$ We can now apply the Sobolev product theorem to the right side to obtain
\begin{align*}
	\|f(u,\nabla u) - f(v,\nabla v)\|_{H^{r-1}} &\leq C\|u+v\|_{H^r}\|u-v\|_{H^r} + C\|v\|_{H^r}^2 \|u-v\|_{H^r}.
\end{align*}
Since $\|v\|_{H^r}\leq R$ by assumption this is bounded by $C \|u-v\|_{H^r} (\|u\|_{H^r} + \|v\|_{H^r}).$

\subsection{Constructing the invariant manifolds: contraction argument}
\label{contraction}

 Let $\Pi_k\colon H^r(\mathbf{S}^n) \to F_k$ be orthogonal projection onto $F_k.$ This orthogonal projection operator is the same for all $r$ because of the way we have defined $H^r.$

Now fix an integer $r\geq 1.$ Define $X_{r,\sigma}$ to be the Banach space of paths $v = v(s)\colon \mathbb R\to H^{r+1}(\mathbf{S}^n)$ for which the norm $\|\cdot \|_{r,\sigma}$ defined by
\begin{align*}
	\|v\|_{r,\sigma} &= \left(\int_0^\infty \|v(s)\|_{H^{r+1}(\mathbf{S}^n)}^2\,\mathrm ds\right)^{1/2} + \sup_{s\geq 0} e^{\sigma s}\|v(s)\|_{H^r(\mathbf{S}^n)}
\end{align*}
is finite. 

We define an operator $T$ for $(v(s),u_0)\in X_{r,\sigma} \times F_k$ by requiring that the path $T(s) = T(v;u_0)(s)$ solve the equation
\begin{align}
	(\partial_s - L)T(s) & = N(v(s)) \label{soln-op} \\
	T(0) &= u_0 - \int_0^\infty e^{-L\tau}(1-\Pi_k) N(v(\tau))\,\mathrm d\tau. \nonumber
\end{align}
The integral in the second equation makes sense pointwise because $1-\Pi_k$ projects on a finite-dimensional invariant subspace for $L.$ We will see moreover that for $N$ satisfying our requirements it is convergent and defines an element of $H^r$ for $v\in X_{r,\sigma}$ with $\lambda_{k-1}<\sigma <\lambda_k.$

Notice that if $v$ is a fixed point for $T(\cdot;u_0),$ then $v$ solves the nonlinear evolution equation (\ref{rMCF}). If this fixed point lies in the space $X_{r,\sigma},$ then by definition it converges to zero exponentially. We will show that for small enough $u_0\in F_k$ and for $\lambda_{k-1}<\sigma<\lambda_k,$ the mapping $T(\cdot;u_0)$ has precisely one fixed point $v$ in a small ball centered at the origin in $X_{r,\sigma}.$ This fixed point depends smoothly in $H^r$ on the parameter $u_0,$ and the initial datum of the corresponding evolution is $v(0).$ The orthogonal projection of $v(0)$ onto $F_k$ is just $u_0,$ and it follows easily that the space of initial data in a small ball of $H^r$ centered at $0$ which converges to zero exponentially with rate between $\lambda_{k-1}$ and $\lambda_k$ is a graph over $F_k.$ The size of the ball in $H^r$ on which this is true depends on the exponential rate $\sigma\in (\lambda_{k-1},\lambda_k),$ but since the solution converges to zero and therefore enters every ball centered at zero it is in fact true that the exponential rate of convergence to zero is automatically better than $\sigma$ for any $\sigma<\lambda_k.$

The main result of this section is the following theorem.
\medskip
\begin{theorem}
\label{contraction-thm}
If $r>n/2+1$ and $\lambda_{k-1}<\sigma<\lambda_k$ and if $u_0\in F_k$ with $\|u_0\|_{H^r}$ sufficiently small, then $T(\cdot;u_0)$ maps a small ball centered at the origin in $X_{r,\sigma}$ into itself and satisfies
\begin{align}
	\|T(v,u_0) - T(w,u_0)\|_{r,\sigma} \leq C \left(\|v\|_{r,\sigma} + \|w\|_{r,\sigma}\right) \|v-w\|_{r,\sigma} \label{cont-bd}
\end{align}
for some constant $C= C(r,\sigma,k)$ depending on $r,\sigma,$ and $k.$
\end{theorem}
\medskip
\begin{corollary}
	 The mapping $T$ is a contraction mapping of a small ball centered at the origin in $X_{r,\sigma}$ into itself. Consequently, it has a unique fixed point in this ball.
\end{corollary}

\begin{proof}[Proof of Theorem \ref{contraction}]
We first prove the bound (\ref{cont-bd}) on a small ball, and then we show that if this ball is small enough it is mapped into itself by $T.$ If $v$ and $w$ are in $X_{r,\sigma}$ and $u_0 \in F_k,$ then the difference $D(s) = T(v;u_0)(s) - T(w;u_0)(s)$ is continuously differentiable and satisfies the equation
\begin{align}
	(\partial_s - L)D(s) &= N(v(s)) - N(w(s)) \label{diff-eq} \\
		D(0) &= -\int_0^\infty e^{-L\tau}(1-\Pi_k) \left(N(v(\tau)) - N(w(\tau))\right)\,\mathrm d\tau.\nonumber
\end{align}
To bound $D,$ we break it up into components using the orthogonal projection $\Pi_k\colon H^r\to F_k.$ The bound on the component $(1-\Pi_k)D(s)$ is simple, so we take care of that first. The more interesting bound is on $\Pi_kD(s),$ and for this we make use of Corollaries \ref{L2-bd} and \ref{sup-bd}, which apply because $\Pi_kD(s)\in F_k$ for all $s\geq 0$ (this is why we break $D$ into components in the first place).

We now show how $(1-\Pi_k)D(s)$ is controlled in $X_{r,\sigma}.$ First, $1-\Pi_k$ projects onto a finite-dimensional subspace of $H^r,$ and $(1-\Pi_k)D(s)$ can be expressed as an integral
\begin{align*}
	(1-\Pi_k)D(s) &= e^{Ls}(1-\Pi_k)D(0) + \int_0^s e^{L(s-\tau)}\left(1-\Pi_k\right)\left(N(v(\tau)) - N(w(\tau))\right)\,\mathrm d\tau\\
	&= -\int_s^\infty e^{L(s-\tau)} \left(1-\Pi_k\right)\left(N(v(\tau)) - N(w(\tau))\right)\,\mathrm d\tau,
\end{align*}
where the second line is obtained from the first by substituting the expression for $D(0)$ and simplifying. For $\tau>s,$ the operator $e^{L(s-\tau)}$ has norm $e^{\lambda_{k-1}(\tau-s)}$ on $\on{range}(1-\Pi_k).$ Because the range is finite-dimensional, and all norms on it are equivalent, we may write
\begin{align*}
	\|(1-\Pi_k) D(s)\|_{H^r(\mathbf{S}^n)} &\leq C\int_s^\infty e^{\lambda_{k-1}(\tau-s)} \|(1-\Pi_k)  \left(N(v(\tau)) - N(w(\tau))\right)\|_{H^{r-1}(\mathbf{S}^n)}\,\mathrm d\tau \\
	&\leq C \int_s^\infty e^{\lambda_{k-1}(\tau-s)} \|N(v(\tau)) - N(w(\tau))\|_{H^{r-1}(\mathbf{S}^n)}\,\mathrm d\tau
\end{align*}
where $C$ is a constant that depends on $k$ and $r.$ Now we just use the nonlinear estimate Lemma \ref{bilin-bd} to bound the right side and obtain
\begin{align*}
	 \|(1-\Pi_k)D(s)\|_{H^r} \leq C\int_s^\infty e^{\lambda_{k-1}(\tau-s)} \left(\|v(\tau)\|_{H^{r+1}}\|v(\tau) - w(\tau)\|_{H^{r}} +\|w(\tau)\|_{H^r} \|v(\tau) - w(\tau)\|_{H^{r+1}}\right)\, \mathrm d\tau.
\end{align*}
Finally, assuming $\lambda_{k-1}<\sigma<\lambda_k,$ we bound the right side by the $\|\cdot\|_{r,\sigma}$ norm straightforwardly as follows (using the first summand for an example):
\begin{align*}
	\int_s^\infty e^{\lambda_{k-1}(\tau-s)} \|v(\tau)\|_{H^{r+1}} &\|v(\tau) - w(\tau)\|_{H^{r}}\,\mathrm d\tau \\
	& \leq \sup_{\tau\geq s} e^{\sigma (\tau-s)} \|v(\tau) - w(\tau)\|_{H^r} \int_s^\infty e^{-(\sigma-\lambda_{k-1})(\tau-s)}\|v(\tau)\|_{H^{r+1}}\,\mathrm d\tau \\
	&\leq e^{-\sigma s}\|v-w\|_{r,\sigma} \left(\int_0^\infty e^{-2(\sigma-\lambda_{k-1})\tau }\,\mathrm d\tau \right)^{1/2}\left(\int_0^\infty \|v(\tau)\|_{H^{r+1}}^2\,\mathrm d\tau\right)^{1/2} \\
	&\leq e^{-\sigma s} \|v-w\|_{r,\sigma} \|v\|_{r,\sigma} \frac{1}{(\sigma - \lambda_{k-1})^{1/2}}.
\end{align*}
The passage from the first to the second line is just Cauchy--Schwarz. All told, we obtain
\begin{align}
	e^{\sigma s} \|(1-\Pi_k)D(s)\|_{H^r} \leq C(\|v\|_{r,\sigma} + \|w\|_{r,\sigma})\|v-w\|_{r,\sigma}, \label{proj-bd1}
\end{align}
where $C$ depends on $k$ and $\sigma.$

Since the $H^r$ and $H^{r+1}$ norms are equivalent on the range of $1-\Pi_k,$ we see from the bound (\ref{proj-bd1}) that
\begin{align*}
	\|(1-\Pi_k)D(s)\|_{H^{r+1}} \leq e^{-\sigma s}C(\|v\|_{r,\sigma} + \|w\|_{r,\sigma})\|v-w\|_{r,\sigma},
\end{align*}
and since $e^{-\sigma s}$ is square-integrable over $[0,\infty)$ for $\sigma>0$ we obtain
\begin{align}
	\int_0^\infty \|(1-\Pi_k)D(s)\|_{H^{r+1}}^2 \,\mathrm ds &\leq C(\|v\|_{r,\sigma} + \|w\|_{r,\sigma})\|v-w\|_{r,\sigma}. \label{proj-bd2}
\end{align}
Combining (\ref{proj-bd1}) and (\ref{proj-bd2}) gives the desired bound
\begin{align*}
	\|(1-\Pi_k)D\|_{r,\sigma} \leq C(\|v\|_{r,\sigma} + \|w\|_{r,\sigma})\|v-w\|_{r,\sigma},
\end{align*}
with $C$ depending on $k$ and $\sigma$ and $r.$

\medskip
Let us now bound $\|\Pi_k D(s)\|_{r,\sigma}.$ Notice that $\Pi_kD(0) = 0,$ so that Corollary \ref{sup-bd} implies
\begin{align*}
	e^{2\sigma s}\left\|\Pi_k D(s)\right\|_{H^r(\mathbf{S}^n)}^2 &\leq\frac{\lambda_k}{2(\lambda_k - \sigma)} \int_0^s e^{2\sigma\tau} \|\Pi_k\left[ N(v(\tau)) - N(w(\tau))\right]\|_{H^{r-1}(\mathbf{S}^n)}^2\,\mathrm d\tau \\
	& \leq\frac{\lambda_k}{2(\lambda_k - \sigma)} \int_0^s e^{2\sigma\tau} \|N(v(\tau)) - N(w(\tau))\|_{H^{r-1}(\mathbf{S}^n)}^2\,\mathrm d\tau.
\end{align*}
To pass from the first line to the second we just use the fact that $\Pi_k$ does not increase the $H^{r-1}$ norm. Inserting the bilinear estimate for $N$ into this we bound the integral as
\begin{align*}
	\int_0^s e^{2\sigma\tau} & \|N(v(\tau)) - N(w(\tau))\|_{H^{r-1}}^2\,\mathrm d\tau \\
	& \leq C \int_0^s e^{2\sigma \tau}\left( \|v(\tau)\|_{H^{r+1}}^2 \|v(\tau)-w(\tau)\|_{H^r}^2 + \|w(\tau)\|_{H^r}^2 \|v(\tau)-w(\tau)\|_{H^{r+1}}^2\right)\,\mathrm d\tau,
\end{align*}
from which, using the definition of $\|\cdot \|_{r,\sigma},$ we straightforwardly obtain
\begin{align*}
	\int_0^s e^{2\sigma\tau} \|N(v(\tau)) - N(w(\tau))\|_{H^{r-1}}^2\,\mathrm d\tau &\leq C\left(\|v\|_{r,\sigma}^2 + \|w\|_{r,\sigma}^2\right)\|v-w\|_{r,\sigma}^2.
\end{align*}
Combining this with the $H^r$ estimate for $D(0)$ we get
\begin{align*}
	e^{2\sigma s}\left\|D(s)\right\|_{H^r(\mathbf{S}^n)}^2 &\leq C\left(1 + \frac{\lambda_k}{\lambda_k - \sigma} \right) \left(\|v\|_{r,\sigma}^2 + \|w\|_{r,\sigma}^2\right)\|v-w\|_{r,\sigma}^2.
\end{align*}

By Corollary \ref{L2-bd} and an analogous use of the nonlinear estimate, we similarly obtain
\begin{align*}
	\int_0^\infty \left\|\Pi_k D(s)\right\|_{H^{r+1}}^2\,\mathrm ds & \leq \int_0^\infty\|N(v(s)) - N(w(s))\|_{H^{r-1}}^2\,\mathrm ds \leq C \left(\|v\|_{r,\sigma}^2 + \|w\|_{r,\sigma}^2\right)\|v-w\|_{r,\sigma}^2.
\end{align*}
This completes the bound on $\|\Pi_k D(s)\|_{r,\sigma}.$

Combining all of these estimates gives us the final bound:
\begin{align*}
	\|D\|_{r,\sigma} \leq \|(1-\Pi_k)D\|_{r,\sigma} + \|\Pi_kD\|_{r,\sigma} &\leq C\left( \frac{\lambda_k}{\lambda_k - \sigma}\right)^{1/2} \left(\|v\|_{r,\sigma} + \|w\|_{r,\sigma}\right)\|v-w\|_{r,\sigma}.
\end{align*}
This proves (\ref{cont-bd}).

\medskip
Now let us show that $T$ maps a small ball centered at the origin in $X_{r,\sigma}$ into itself. Let $U(s) = e^{Ls}u_0$ be the solution to the linear homogeneous equation $(\partial_s - L) U = 0$ with initial data $U(0) = u_0.$ First, taking $w=0$ in (\ref{cont-bd}) shows, since $T(0;u_0) = U$ by the definition (\ref{soln-op}) of $T,$ that
\begin{align*}
	\|T(v;u_0) - U\|_{r,\sigma} &\leq C\|v\|_{r,\sigma}^2.
\end{align*}
Therefore if $0<\delta<1/C$ and $\|U\|_{r,\sigma}<\delta - C\delta^2,$ then $\|T(v;u_0)\|_{r,\sigma}<\delta$ whenever $\|v\|_{r,\sigma}<\delta.$ That is, $T(\cdot;u_0)$ maps the ball of radius $\delta$ centered at zero in $X_{r,\sigma}$ into itself. We need only to show now that $\|U\|_{r,\sigma}$ can be controlled by $\|u_0\|_{H^r(\mathbf{S}^n)}.$ But this follows immediately from the estimates of Corollaries \ref{sup-bd} and \ref{L2-bd} since $U(s)\in F_k$ for all $s\geq 0.$
\end{proof}

\section{Asymptotics of the limit}
\label{asymptotics}

In the preceding section, we constructed, for each $k\geq 2,$ a codimension $d_k$ invariant submanifold for equation (\ref{rMCF}) consisting of solutions which converge to zero with exponential rate $\sigma$ for every $\sigma<\lambda_k.$ In this section, we show that any such solution must actually converge to zero with exponential rate $\lambda_k,$ and we show also that any such solution is approximated well by a solution to the linear equation.

The first result is the following. 

\medskip
\begin{proposition}
\label{rate}
	Suppose $k\geq 2$ is an integer and $u\colon \bf{S}^n\times[s_0,\infty)\to\mathbb R$ is a solution to (\ref{rMCF}) which satisfies
	\begin{align*}
		\sup_{s\geq s_0} e^{\sigma s} \|u(s)\|_{H^{r+2}(\bf{S}^n)} <\infty
	\end{align*}
	for all $\sigma <\lambda_k.$ Then
	\begin{align*}
		\sup_{s\geq s_0} e^{\lambda_k s} \|u(s)\|_{H^r(\bf{S}^n)} <\infty
	\end{align*}
	and in fact there exists $P\in E_k$ such that
	\begin{align*}
		u(y,s) &= e^{-\lambda_k s} P(y) + O\left(e^{-2\lambda_k s} + e^{-\lambda_{k+1}s}\right)
	\end{align*}
	in $H^r(\bf{S}^n)$ as $s\to\infty.$
\end{proposition}

We now prove a lemma, showing that the first hypothesis of Proposition \ref{rate} is met automatically for all solutions of (\ref{rMCF}) satisfying 
\begin{align*}
	\sup_{s\geq s_0} e^{\sigma s} \|u(s)\|_{H^{r}(\bf{S}^n)} <\infty
\end{align*}
for all $\sigma <\lambda_k.$ Notice that Proposition \ref{rate} requires this condition to hold for the $H^{r+2}(\bf{S}^n)$ norm, and not just the $H^{r}(\bf{S}^n)$ norm. We show that it always holds in the $H^{r+2}(\bf{S}^n)$ norm if it holds in the $H^{r}(\bf{S}^n)$ norm.

\medskip
\begin{lemma}
	Suppose $u\colon \bf{S}^n\times [s_0,\infty)\to\mathbb R$ is a solution to (\ref{rMCF}) converging to zero in $L^2(\bf{S^n})$ as $s\to\infty.$ Then either $u$ is identically zero or
	\begin{align*}
		\sup_{s \geq s_0} \frac{\|u(s)\|_{H^{r+1}(\bf{S}^n)}}{\|u(s)\|_{H^r(\bf{S}^n)}} <\infty
	\end{align*}
	for every integer $r\geq 0.$
\end{lemma}
\begin{proof}
The crucial feature of rescaled mean curvature flow making this work is that a solution $u$ to (\ref{rMCF}) converging to zero in $L^2(\bf{S}^n)$ also converges to zero in $H^r(\bf{S}^n)$ for every $r\geq 0.$ This follows from Huisken's result, \cite{Hu84} (see Remark (i) after Theorem 1.1), that convergence of a convex mean curvature flow to the sphere is exponential in $C^k$ for any $k.$ The rest of the proof uses generalities about the equation (\ref{rMCF}) satisfied by $u.$ 

Since $u$ converges to zero in $H^r(\bf{S}^n)$ for every $r,$ it lies in one of the invariant manifolds of Theorem \ref{stable-mfld}, as proved in the preceding section. Moreover, $u$ cannot converge to zero faster than any exponential unless it is identically zero, as proved in \cite{Str18} (see Theorem 2.2). Therefore, if $u$ is not identically zero, there is a largest integer $k = k(r)\geq 2,$ depending on $r,$ with the property that
\begin{align*}
	\sup_{s\geq s_0} e^{\sigma s} \|u(s)\|_{H^r(\bf{S}^n)} <\infty
\end{align*}
for all $\sigma <\lambda_k.$ Since $\|\cdot\|_{H^{r+1}(\bf{S}^n)}\geq \|\cdot\|_{H^{r}(\bf{S}^n)},$ the integer $k(r)$ does not increase with $r.$ This means that eventually $k(r)$ is constant in $r,$ that is, there exists some $r_0$ such that $k(r) = k(r_0)$ for $r\geq r_0.$ 

Then for $r\geq r_0$ we can apply Proposition \ref{rate} to conclude that
\begin{align*}
	\sup_{s\geq s_0} e^{\lambda_k s} \|u(s)\|_{H^r(\bf{S}^n)}<\infty,
\end{align*}
where $k = k(r) = k(r_0),$ and that there exists $P\in E_k$ with the property that $\|e^{\lambda_k s} u(s) - P\|_{H^r(\bf{S^n})} \leq Ce^{-\sigma s}$ for some $\sigma>0.$

Now $P$ must be nonzero, otherwise $e^{\sigma s}\|u(s)\|_{H^r(\bf{S}^n)}$ would be bounded for all $\sigma<\lambda_{k+1},$ and $k$ would not be the largest integer with this property. (It now follows easily that $k=k(r)$ is the same for all $r\geq 0$ and not just all sufficiently large $r.$)

This is enough to conclude that $\|u(s)\|_{H^{r+1}(\bf S^n)}/\|u(s)\|_{H^r(\bf S^n)}$ is bounded in $s$ for any $r,$ since
\begin{align*}
	\frac{\|u(s)\|_{H^{r+1}(\bf S^n)}}{\|u(s)\|_{H^r(\bf S^n)}} \leq \frac{e^{-\lambda_ks} \|P\|_{H^{r+1}(\bf S^n)} + C_1 e^{-(\lambda_k + \sigma)s}}{ e^{-\lambda_ks} \|P\|_{H^{r}(\bf S^n)} - C_2 e^{-(\lambda_k + \sigma)s}}
\end{align*}
for some constants $C_1,C_2,\sigma>0,$ and for $s$ sufficiently large the expression on the right is bounded.
\end{proof}

The proof of Proposition \ref{rate}, to which we now turn, will be the consequence of a series of three lemmas in which we bound the projections of $u(s)$ onto $F_{k+1} = \bigoplus_{j\geq k+1} E_j,$ onto $E_k,$ and onto $F_k^\perp = \bigoplus_{j<k} E_j,$ with $E_j$ as always the eigenspace for $\Delta+1$ corresponding to $\lambda_j.$

We begin by bounding the projection onto $F_{k+1}.$

\medskip
\begin{lemma}
\label{st-bd}
	Under the hypotheses of Proposition \ref{rate},
	\begin{align*}
		\|\Pi_{k+1} u(s)\|_{H^r(\bf{S}^n)} = O(e^{-\lambda_{k+1}s} + e^{-2\sigma s})
	\end{align*}
	as $s\to\infty,$ for any $\sigma <\lambda_k.$
\end{lemma}
\begin{proof}
Notice that
\begin{align*}
	\|\Pi_{k+1} u(s)\|_{H^r} \frac{\mathrm d}{\mathrm ds} \|\Pi_{k+1}u(s)\|_{H^r} &= \frac{\mathrm d}{\mathrm ds} \|\Pi_{k+1}u(s)\|_{H^r}^2/2 \\
	&= \langle \Pi_{k+1} u(s),Lu(s)\rangle_{H^r} + \langle \Pi_{k+1}u(s),N(u(s))\rangle_{H^r} \\
	&\leq -\lambda_{k+1} \|\Pi_{k+1}u(s)\|_{H^r}^2 + \|\Pi_{k+1}u(s)\|_{H^r} \|N(u(s))\|_{H^r}.
\end{align*}
Thus if $\mu\leq \lambda_{k+1},$ we get
\begin{align}
	\frac{\mathrm d}{\mathrm ds} \left(e^{\mu s} \|\Pi_{k+1} u(s)\|_{H^r}\right) &\leq e^{\mu s}\|N(u(s))\|_{H^r}. \label{st-bd1}
\end{align}
On the other hand 
\begin{align}
	\|N(u(s))\|_{H^r} \leq C\|u(s)\|_{H^{r+1}}\|u(s)\|_{H^{r+2}} \leq C e^{-2\sigma s} \label{st-bd2}
\end{align}
for any $\sigma <\lambda_k$ (the constant may depend on $\sigma$). The first inequality is the nonlinear estimate Lemma \ref{bilin-bd} and the second follows from the hypothesis of Proposition \ref{rate}.

Inserting (\ref{st-bd2}) into (\ref{st-bd1}) and integrating shows that $e^{\mu s} \|\Pi_{k+1}u(s)\|_{H^r}$ is bounded provided $\mu\leq \lambda_{k+1}$ and $\mu<2\lambda_k.$ This is the same as the conclusion of the lemma.
\end{proof}

Next we bound the projection onto $F_k^\perp.$

\medskip
\begin{lemma}
\label{unst-bd}
	In the situation of Proposition \ref{rate},
	\begin{align*}
		\|(1-\Pi_k) u(s)\|_{H^r(\bf{S}^n)} = O( e^{-2\sigma s})
	\end{align*}
	as $s\to\infty,$ for any $\sigma <\lambda_k.$
\end{lemma}
\begin{proof}
	Reasoning as in the proof of Lemma \ref{st-bd}, we obtain
	\begin{align*}
		\frac{\mathrm d}{\mathrm ds} \|(1-\Pi_k) u(s)\|_{H^r(\bf{S}^n)} &\geq -\lambda_{k-1}\|(1-\Pi_k) u(s)\|_{H^r(\bf{S}^n)} - C\|N(u(s))\|_{H^r} \\
		&\geq - \lambda_{k-1} \|(1-\Pi_k) u(s)\|_{H^r(\bf{S}^n)} - C e^{-2\sigma s}
	\end{align*}
	for any $\sigma <\lambda_k.$ In other words,
	\begin{align*}
		\frac{\mathrm d}{\mathrm ds}\left( e^{\lambda_{k-1}s} \|(1-\Pi_k) u(s)\|_{H^r}\right) &\geq -C e^{(\lambda_{k-1} - 2\sigma) s}.
	\end{align*}
	Integrating this gives, for $s<t$ and $\sigma >\lambda_{k-1}/2,$
	\begin{align*}
		e^{\lambda_{k-1} t} \|(1-\Pi_k) u(t)\|_{H^r} &\geq  e^{\lambda_{k-1} s} \|(1-\Pi_k) u(s)\|_{H^r} -C e^{(\lambda_{k-1} - 2\sigma) s} \\
		&= e^{\lambda_{k-1} s} \left(\|(1-\Pi_k) u(s)\|_{H^r} - C e^{-2\sigma s}\right).
	\end{align*}
	Now make $t\to\infty.$ Because $\lambda_{k-1}<\lambda_k,$ the hypothesis of Proposition \ref{rate} tells us that the left side converges to zero. Consequently the right side must be non-positive, or, in other words,
	\begin{align*}
		\|(1-\Pi_k) u(s)\|_{H^r} \leq C e^{-2\sigma s}
	\end{align*}
	for all $s\geq 0$ and any $\sigma <\lambda_k.$ As far as we know, $C$ depends on $\sigma$ of course.
\end{proof}

Finally, we turn to the projection onto $E_k.$ Write $\pi_k = \Pi_k - \Pi_{k+1}$ for the projection of $H^r$ onto $E_k.$

\medskip
\begin{lemma}
\label{ctr-bd}
	Under the hypotheses of Proposition \ref{rate} and assuming $s<t,$
	\begin{align*}
		\| e^{\lambda_k t}\pi_ku(t) - e^{\lambda_k s}\pi_k u(s)\|_{H^r(\mathbf{S}^n)} &\leq C e^{-\sigma s}
	\end{align*}
	for any $\sigma <\lambda_k.$
\end{lemma}
\begin{proof}
First,
\begin{align*}
	\frac{\mathrm d}{\mathrm ds} e^{\lambda_k s} \pi_k u(s) &= e^{\lambda_k s}\pi_k N(u(s))
\end{align*}
since $L\pi_k u = -\lambda_k \pi_k u.$ Now we integrate this (the equation is in a finite dimensional vector space, namely, the range of $\pi_k$) and use the triangle inequality to obtain
\begin{align*}
	\| e^{\lambda_k t}\pi_ku(t) - e^{\lambda_k s}\pi_k u(s)\|_{H^r} &= \left\| \int_s^t e^{\lambda_k s}\pi_k N(u(\tau))\,\mathrm d\tau \right\|_{H^r} \\
	&\leq \int_s^t e^{\lambda_k s}\|\pi_kN(u(\tau))\|_{H^r}\,\mathrm d\tau \\
	&\leq C \int_s^t e^{(\lambda_k- 2\sigma )\tau}\,\mathrm d\tau
\end{align*}
for $s<t$ and any $\sigma<\lambda_k.$ This implies the lemma.
\end{proof}

An immediate corollary of Lemma \ref{ctr-bd} is that $e^{\lambda_k s} \pi_k u(s)$ converges in $H^r$ norm exponentially fast. The limit of course must be an element of $E_k,$ that is, an eigenfunction $P$ of $\Delta +1$ with eigenvalue $\lambda_k.$

Altogether, Lemmas \ref{st-bd}, \ref{unst-bd}, and \ref{ctr-bd} imply that
\begin{align*}
	u(y,s) &= e^{-\lambda_k s} P(y) + O\left(e^{-\lambda_{k+1}s} +e^{-2\sigma s}\right)
\end{align*}
in $H^r(\bf{S}^n)$ as $s\to\infty,$ for any $\sigma <\lambda_k.$ In particular, $\|u(s)\|_{H^r} \leq C e^{-\lambda_k s}$ for some $C>0,$ and this fact can be used to improve the asymptotics and take $\sigma = \lambda_k$ in Lemmas \ref{unst-bd} and \ref{ctr-bd} (but not Lemma \ref{st-bd}) as follows. The appearance of $\sigma$ came through $H^r$ bounds on the nonlinear term $N(u(s))$ used in the proofs of Lemmas \ref{st-bd}, \ref{unst-bd}, and \ref{ctr-bd}: We could only say, based on the hypotheses of Proposition \ref{rate}, that $\|N(u(s))\|_{H^r} \leq C e^{-2\sigma s}$ for any $\sigma<\lambda_k$ and for some $C>0$ depending on $\sigma.$ But now that we know $\|u(s)\|_{H^r}$ (hence $\|u(s)\|_{H^{r+1}}$ and $\|u(s)\|_{H^{r+2}}$ by standard parabolic estimates using the fact that $u$ is a solution to (\ref{rMCF})) is actually $O(e^{-\lambda_k s}),$ we can say that $N(u(s)) = O(e^{-2\lambda_k s})$ and then obtain improvements on the error in Lemmas \ref{unst-bd} and \ref{ctr-bd}. 

A close look at the proof of Lemma \ref{st-bd} reveals that the same method does not work there and we are stuck with the appearance of $\sigma$ in the conclusion. Of course it does not matter so long as $2\lambda_k>\lambda_{k+1},$ which is the case for most $k.$

We summarize these observations in a corollary, which states a more precise version of Proposition \ref{rate}.

\medskip
\begin{corollary}
\label{precise-rate}
	Under the hypotheses of Proposition \ref{rate}:
	\begin{itemize}
		\item[(a)] The projection $\Pi_{k+1}u$ onto the sum $F_{k+1}$ of eigenspaces corresponding to eigenvalues $\lambda_j$ with $j>k$ satisfies
		\begin{align*}
			\|\Pi_{k+1}u(s)\|_{H^r(\bf{S}^n)} \leq C\left(e^{-\lambda_{k+1} s} + e^{-2\sigma s}\right)
		\end{align*}
		for any $\sigma <\lambda_k$ and some $C>0$ (depending on $\sigma$) and all $s\geq s_0.$
		\item[(b)] The projection $(1-\Pi_k)u$ onto the sum $F_k^\perp$ of eigenspaces corresponding to eigenvalues $\lambda_j$ with $j<k$ satisfies
		\begin{align*}
			\|(1-\Pi_k) u(s)\|_{H^r(\bf{S}^n)} \leq C e^{-2\lambda_k s}
		\end{align*}
		for some $C>0$ and all $s\geq s_0.$
		\item[(c)] The projection $\pi_k u$ onto the eigenspace $E_k$ corresponding to the eigenvalue $\lambda_j$ satisfies
		\begin{align*}
			\|e^{\lambda_k s} \pi_k u(s) - P\|_{H^r(\bf{S}^n)} \leq C e^{-\lambda_k s}
		\end{align*}
		for some $P\in E_k$ and some $C>0$ and all $s\geq s_0.$
	\end{itemize}
\end{corollary}

\bigskip
We now obtain more precise asymptotics. 

\medskip
\begin{lemma}
\label{higher-order}
	Suppose $u$ satisfies the hypotheses of Proposition \ref{rate}, and suppose $j\geq k$ and $\lambda_j <2\lambda_k.$ Then there exists $P_j $ in the eigenspace $E_j$ corresponding to $\lambda_j$ such that
	\begin{align*}
		\|e^{\lambda_j s}\pi_j u(s) - P_j\|_{H^r(\bf{S}^n)} \leq C e^{(\lambda_j - 2\lambda_k)s}
	\end{align*}
	for all $s\geq s_0.$
\end{lemma}
\begin{proof}
We argue as in Lemma \ref{ctr-bd}, using
\begin{align*}
	\frac{\mathrm d}{\mathrm ds} e^{\lambda_j s} \pi_j u(s) = e^{\lambda_j s} \pi_j N(u(s))
\end{align*}
to obtain
\begin{align*}
	\|e^{\lambda_j t} \pi_j u(t) - e^{\lambda_j s} \pi_j u(s)\|_{H^r} &\leq C \int_s^t e^{\lambda_j \tau} \|N(u(\tau))\|_{H^r}\,\mathrm d\tau \leq C \int_s^t e^{(\lambda_j - 2\lambda_k) \tau}\,\mathrm d\tau.
\end{align*}
The right side is $O(e^{(\lambda_j - 2\lambda_k)s})$ independent of $t$ provided that $\lambda_j<2\lambda_k,$ and this gives the conclusion of the lemma.
\end{proof}

From Lemma \ref{higher-order} we obtain higher order asymptotics in certain cases (when $k$ is large).

\medskip
\begin{corollary}
	Let $u$ satisfy the hypotheses of Proposition \ref{rate}. Then for $j\geq k$ such that $\lambda_j<2\lambda_k,$ there exists $P_j \in E_j$ such that
	\begin{align*}
		u(y,s) &= \sum_{\genfrac{}{}{0pt}{2}{j\geq k}{\lambda_j<2\lambda_k}} e^{-\lambda_j s} P_j(y) + O(e^{-2\sigma s})
	\end{align*}
	for any $\sigma<\lambda_k.$
\end{corollary}

\subsection{Prescribing the first-order asymptotics}

In this section we prove Theorem \ref{prescribed-lin}. Given $a\in F_k$ sufficiently small in $H^r$ for $r>n/2+1,$ we constructed in Section \ref{contraction} a unique solution $u(s;a)$ to (\ref{rMCF}) satisfying $\Pi_ku(0;a) = a$ and converging to zero with exponential rate $\lambda_k.$ In the preceding subsection we showed that 
\begin{align*}
	P(a) &= \lim_{s\to\infty} e^{\lambda_k s} u(s;a)
\end{align*}
exists and is an element of the eigenspace $E_k$ corresponding to $\lambda_k.$ Here we will study the map $a\mapsto P(a).$ We will show that the image of this map contains a small ball centered at the origin in $E_k.$ This is enough to conclude that every $P\in E_k$ is attained as the limit $e^{\lambda_k s} u(s)$ of some solution $u$ to (\ref{rMCF}), because we can always replace $u$ with $\tilde{u}(s) = u(s-s_0)$ for $s\geq s_0$ thereby scaling the limit by a factor $e^{\lambda_k s_0}.$

Actually, we do not even need to look at arbitrary $a\in F_k$ to obtain surjectivity: we may even restrict attention to $a\in E_k.$ The precise result is the following:

\medskip
\begin{proposition}
	There exists $\delta>0$ such that if $b\in E_k$ satisfies $\|b\|_r<\delta,$ there exists $a\in E_k$ with $P(a) = b.$
\end{proposition}

\emph{Remark.} A slightly more careful argument along the lines of the below proof shows that $P$ actually maps a small neighborhood of the origin in $E_k$ homeomorphically onto another small neighborhood of the origin.

\begin{proof}
Let us first recall the norm $\|\cdot\|_{r,\sigma}$ from Section \ref{contraction}:
\begin{align*}
	\|v\|_{r,\sigma} &= \left(\int_0^\infty \|v(s)\|_{r+1}^2\,\mathrm ds\right)^{1/2} + \sup_{s\geq 0} e^{\sigma s}\|v(s)\|_r.
\end{align*}
It follows from Theorem \ref{contraction-thm} that if $\sigma<\lambda_k$ and $a\in F_k$ is sufficiently small in $H^r,$ then
\begin{align*}
	\|u(s;a) - e^{Ls}a\|_{r,\sigma} \leq C\|u(\cdot;a)\|_{r,\sigma}^2
\end{align*}
for some constant $C>0$ depending only on $\|a\|_{H^r}.$ By making $a$ smaller if necessary, we may moreover assume that $\|u(\cdot;a)\|_{r,\sigma}<1/(2C)$ so that we get the bound
\begin{align}
	\|u\|_{r,\sigma} \leq 2\|e^{Ls}a\|_{r,\sigma} \leq 4\|a\|_{H^r}. \label{pr.init-bd}
\end{align}
The last inequality is just an $H^r$ estimate for the homogeneous linear equation. Thus we can bound $\|u\|_{r,\sigma}$ by a constant times $\|a\|_{H^r}$ for any $\sigma<\lambda_k,$ provided that $\|a\|_{H^r}$ is small enough.

We now make use of the representation
\begin{align*}
	e^{\lambda_k s} u(s;a) & = e^{(\lambda_k+L)s}a  \\
	&+ \int_{0}^s e^{(\lambda_k + L)(s-t)} e^{\lambda_kt} \Pi_{k} N(u(t;a))\,\mathrm dt - \int_s^\infty e^{(\lambda_k + L)(s-t)}e^{\lambda_k t}(1-\Pi_{k}) N(u(t;a))\,\mathrm dt,
\end{align*}
which is valid for $u$ because  $e^{\lambda_k s}u(s;a)$ is bounded. By taking the $H^r$ norm of both sides and applying the triangle inequality we deduce
\begin{align*}
	\left\|e^{\lambda_k s} u(s) -e^{(\lambda_k + L)s} a\right\|_{H^r} &\leq \int_{0}^\infty e^{\lambda_kt} \|N(u(t))\|_{H^r}\,\mathrm dt \leq C\int_{0}^\infty e^{\lambda_k t} \|u(t)\|_{H^{r+1}}\|u(t)\|_{H^{r+2}} \,\mathrm dt.
\end{align*}
where in the last inequality we've used the nonlinear bound $\|N(u)\|_{H^r}\lesssim \|u\|_{H^{r+1}}\|u\|_{H^{r+2}}$ from Lemma \ref{bilin-bd}. Now, the right side can be bounded by $\|u\|_{r+2,3\lambda_k/4}^2,$ for example, as follows:
\begin{align*}
	\int_{0}^\infty e^{\lambda_k t}\|u(t)\|_{H^{r+1}} \|u(t)\|_{H^{r+2}} \,\mathrm dt &\leq\int_{0}^\infty\left( e^{3\lambda_k t/4}\|u(t)\|_{H^{r+2}}\right)^2e^{-\lambda_k/2} \,\mathrm dt \leq \frac{2}{\lambda_k}\left(\sup_{t\geq 0} e^{3\lambda_k/4}\|u(t)\|_{H^{r+2}}\right)^2.
\end{align*}

If we now assume that $a$ is sufficiently small in $H^{r+2}$ rather than in $H^r$ and employ the bound (\ref{pr.init-bd}) (with $r+2$ instead of $r$) we obtain 
\begin{align*}
	\left\|e^{\lambda_k s} u(s) -e^{(\lambda_k + L)s} a\right\|_{H^r} \leq 4 \|a\|_{H^{r+2}}^2.
\end{align*}
On the left side we let $s\to\infty.$ If $\pi_ka$ is the projection of $a$ onto the eigenspace $E_k,$ the result is
\begin{align*}
	\frac{1}{\lambda_k^2} \|P(a) - \pi_k a\|_{H^{r+2}} = \| P(a) - \pi_ka\|_{H^r} \leq 4\|a\|_{H^{r+2}}^2.
\end{align*}
The first inequality is just the definition of $H^r$ norm on the eigenspace $E_k.$

To finish the argument, we restrict attention to $a\in E_k.$ For such $a$ we have $\pi_k a = a$ and the foregoing estimate reduces to $\|P(a) - a\| \leq C\|a\|^2$ for all sufficiently small $a$ in $E_k$ (the norm is unimportant because $E_k$ is finite-dimensional). This is enough to prove that the image of $P$ contains a small ball in $E_k$ centered at the origin. 

Indeed, we run a standard contraction argument as in one direction of the proof of the inverse function theorem: if $b\in E_k,$ we want to solve the fixed point equation $a = b - (P(a) - a).$ If $\delta <1/(2C)$ and $0<\|b\|\leq \delta - C\delta^2,$ then the map $F$ defined by $F(a) = b-(P(a) - a)$ sends the closed ball $\|a\|\leq\delta$ to itself. Indeed if $\|a\|\leq \delta$ then
\begin{align*}
	\|F(a)\| = \|b- (P(a)-a)\| \leq \|b\| + \|P(a) - a\| \leq \delta - C\delta^2 + C\delta^2 = \delta.
\end{align*}
On the other hand, $F$ depends continuously on $a$ (this follows from the proof of Theorem \ref{contraction-thm}) and so, being a continuous mapping of a closed ball into itself, it must have a fixed point $a,$ that is, a solution to $P(a) = b.$
\end{proof}

\clearpage
\bibliography{/Users/nstrehlke/Documents/MathBibliography}
\bibliographystyle{amsalpha}

\end{document}